\documentclass{amsproc}
\usepackage{graphicx}
\usepackage{amsmath,amssymb}
\usepackage[figuresright]{rotating}
\usepackage{rotating}
\newtheorem{theorem}{Theorem}[section]
\newtheorem{lemm}[theorem]{Lemma}
\newtheorem{prop}[theorem]{Proposition}
\newtheorem{coro}[theorem]{Corollary}

\theoremstyle{definition}
\newtheorem{defi}[theorem]{Definition}

\theoremstyle{remark}
\newtheorem{remark}[theorem]{Remark}
\numberwithin{equation}{section}

\def\om{\omega}

\def\dim{\hbox{dim}}

\newfont{\df}{eufm10}

\def\dim{\hbox{\rm dim}\,}

\begin{document}

\title[Novel isoclasses of one-parameter exotic small quantum groups]
{Novel isoclasses of one-parameter exotic small quantum groups originating from a two-parameter framework}
\author[Hu]{Naihong Hu{$^*$}}
\address{School of Mathematical Sciences, MOE Key Laboratory of Mathematics and Engineering Applications \& Shanghai Key Laboratory of PMMP, East China Normal University, Shanghai 200241, China}
\email{nhhu@math.ecnu.edu.cn}
\thanks{$^*$Corresponding author. This work is
supported by the NNSF of China (Grant No. 12171155), and in part by the Science and Technology Commission of Shanghai Municipality (Grant No. 22DZ2229014).}
\author[Xu]{Xiao Xu}
\address{School of Mathematical Sciences, MOE Key Laboratory of Mathematics and Engineering Applications \& Shanghai Key Laboratory of PMMP, East China Normal University, Shanghai 200241, China}
\email{52275500004@stu.ecnu.edu.cn}

\date{2023.06.09}


\begin{abstract}
The classification of one-parameter small quantum groups remains a fascinating open problem. This paper uncovers a novel phenomenon: beyond the Lusztig small quantum groups—equipped with double group-like elements \cite{Lusztig restricted JAMS 1990}—there exists a plethora of exotic small quantum groups, approximately fivefold more numerous than their standard counterparts, which originate from a two-parameter framework.
\end{abstract}

\maketitle
\section{Introduction}
Over the past two decades, a systematic investigation into two-parameter quantum groups has been undertaken, for instance, see \cite{BW1, BW2} for type $A$, \cite{Bai Xiaotang 2008, BH, BGH 2006, BGH 2007, Chen Hu and Wang to appear,HS} for other finite types; for a unified definition, see \cite{Hu  Pei  Sci. in China Ser A: Math 2008, HP};
\cite{GHZ, Hu Rosso Zhang 2008, HZ1, HZ2, JZ, JZ2} for the affine untwisted types;
\cite{JZ1, JZ3} for the affine twisted types; \cite{Bai Xiaotang 2008, Benkart Pereira Witherspoon 2010, Benkart Witherspoon Fields Institute Communications 2004,R Chen PhD,Chen Hu and Wang to appear, Hu Wang  Pacific J. Math 2009, Hu  Wang  JGP 2010} for the restricted two-parameter quantum groups for some of
finite types. Generalization to multi-parameter setting, see \cite{Pei Hu Rosso 2010}.
 It is the defining approach to multi-parameter quantum groups that leads to a one-parameter new quantum affine algebra $\mathcal U_q(\widehat{\mathfrak{sl}_2})$ (of infinite dimension) to be found in \cite{Feng Hu Zhuang JGP 2021} (we refer to it as the {\it admissible} quantum affine algebra),
 which is not isomorphic to the standard quantum affine algebra $U_q(\widehat{\mathfrak{sl}_2})$ as Hopf algebras in the case when $q$ is generic.

Against this backdrop, the present work focuses on two-parameter small (i.e., restricted) quantum groups in scenarios where both parameters $r$ and $s$ are roots of unity.
Our analysis yields striking classification results—even in the one-parameter case—that surpass prior expectations.

For the two-parameter small (i.e., restricted) quantum groups $\mathfrak u_{r,s}(\mathfrak{sl}_3)$ with $r, s$ of order $4$, in this case they have the Drinfeld double structure $\mathcal{D}(H)$.
Let us recall a work due to Benkart-Pereira-Witherspoon (\cite{Benkart Pereira Witherspoon 2010}),
in which they successively used some techniques such as: the Hopf $2$-cocycle twist-equivalences (Theorem 4.12 \cite{Benkart Pereira Witherspoon 2010}),
the Isomorphism Reduction Theorem of two-parameter small quantum groups of type $A$ due to the $1$st author and his coauthor (cf. Theorem 5.5 \cite{Hu  Wang  JGP 2010}), and a well-known result due to Radford
\cite{Radford construction J Algebra 2003} (see Theorem \ref{thm2.11}) and counting the dimension distributions of simple Yetter-Drinfeld $H$-modules in order to figure out a new isoclass, a one-parameter exotic small quantum group $\mathfrak u_{1,q}(\mathfrak{sl}_3)$, which is not isomorphic to the standard one $\mathfrak u_{q,q^{-1}}(\mathfrak{sl}_3)$ as pointed Hopf algebras.
This intrigues our curiosity to further probe the question:
How many isoclasses of one-parameter exotic (i.e., non-standard) small quantum groups that are of triangular decompositions but not isomorphic to the one-parameter standard ones (i.e., the Lusztig small ones)  can be derived from the two-parameter small quantum groups via the Isomorphsim Reduction Theorems (established by the $1$st author and his coauthors in \cite{Bai Xiaotang 2008, R Chen PhD, Chen Hu and Wang to appear,  Hu Wang  Pacific J. Math 2009, Hu  Wang  JGP 2010})?

The goal of this paper is to classify and enumerate the isoclass representatives as described in the Abstract with explicit presentations.

In order to yield a neat classification for the one-parameter exotic small quantum groups in the root of unity cases that arise from a two-parameter framework, we improve and obtain a refinement version of the Isomorphism Reduction Theorem aforementioned (see Theorem \ref{thm refineA}).
Building on this refined theorem and the PBW-Lyndon Basis Theorem, we derive a complete classification that applies not only to cases with low ranks and low orders but also to \(\mathfrak{u}_{r, s}(\mathfrak{g})\) of general rank and order.
In fact, we find that the order of $rs^{-1}$ is a sufficient strong invariant when classifying the isoclasses of $\mathfrak{u}_{r, s}(\mathfrak{g})$ with the nonprime order $\ell$ of $q$ (for instance, see Propositions \ref{prop A 6} and \ref{prop A 8}).
(Comparably, $\textrm{lcm}(|r|, |s|)$ is a weak invariant, as it corresponds to the dimension of $\mathfrak{u}_{r, s}(\mathfrak{g})$).
While for odd prime order $\ell$ of parameters $r, s$ (as roots of unity),
these restricted two-parameter quantum groups are of the Drinfeld double structures,
we use the Isomorphism Reduction Theorem to yield the complete classification lists of one-parameter small quantum groups of the Drinfeld quantum doubles (as pointed Hopf algebras) for types $A, B, C, D, F_4, G_2$ (see Theorems \ref{thm A general}, \ref{thm B general}, \ref{thm C general}, \ref{thm D general}, \ref{thm F general} and \ref{thm G general}).

As a valid verification of our classification results we obtained above, in the final section, we continue to adopt the representation-theoretic approach as used in \cite{Benkart Pereira Witherspoon 2010}.
As we have known, when those two-parameter small quantum groups are of Drinfeld double structure,
$H=\mathfrak u_{r,s}(\mathfrak b)$, the quantum Borel subgroup of $\mathfrak u_{r,s}(\mathfrak g)$,
their Yetter-Drinfeld $H$-module category $_H\mathcal {YD}^H$ is isomorphic to the category of $D(H)$-modules (\cite{Kas}, \cite{Majid 1991}).
Namely, we can distinguish them by differences in the dimension distributions of their simple modules (equivalently, simple Yetter-Drinfeld modules).
Assume $r=q^y, s=q^z$, where $q$ is a primitive $4$th root of unity.
Precisely, we can proceed with the following steps.
Firstly, we find that $\mathfrak{u}_{1, q}(\mathfrak{sl}_3),
\mathfrak{u}_{q, q^2}(\mathfrak{sl}_3)$ and $\mathfrak{u}_{q, q^3}(\mathfrak{sl}_3)=\mathfrak{u}_{q, q^{-1}}(\mathfrak{sl}_3)$ (the standard one) have Drinfeld double structure. Secondly, using the Hopf $2$-cocycle twist equivalence theorem,
we conclude that $\mathfrak{u}_{1, q}(\mathfrak{sl}_3)$ and  $ \mathfrak{u}_{q, q^2}(\mathfrak{sl}_3)$ have the same dimension distribution.
Finally, Benkart-Pereira-Witherspoon separated  $\mathfrak{u}_{1, q}(\mathfrak{sl}_3), \mathfrak{u}_{q, q^3}(\mathfrak{sl}_3)=\mathfrak{u}_{q, q^{-1}}(\mathfrak{sl}_3)$ by calculating the respective dimension distributions of simple modules.
Moreover, from the perspective of representation theory, we can also find many exotic small quantum groups when $q$ is $6$th or $8$th root of unity, apart from the one-parameter Lusztig small quantum group $\mathfrak{u}_{q, q^{-1}}(\mathfrak{sl}_3)$.

It is worthy to be mentioned that when the parameters as roots of unity have lower order:
for type $A$\;(with parameters of order $4, 5, 6, 7, 8$),
especially for type $A_2$\;(with parameters of order $4, 6, 8$),
for type $B, C, D,  F_4$\;(with parameters of order $5, 7$),
for type $G_2$ (with parameters of order $5, 7, 8$),
we have enumerated correspondingly $209$ isoclasses of new exotic one-parameter small quantum groups with triangular decompositions (most of them have the Drinfeld double structures),
associated to the $45$ isoclasses of one-parameter Lusztig small quantum groups with order $\ell\in\{4, 5, 6, 7, 8\}$. If one takes all types in the
$E$-series into account, the ratio of exotic to standard is approximately $5$.

In the root of unity setting,  this paper exhibits an interesting picture: Beyond the (one-parameter) Lusztig small quantum groups \cite{Lusztig restricted JAMS 1990, Lusztig 1993}, there exist numerous exotic one-parameter small quantum groups. The first author gratefully acknowledges Prof. Molev for posing a pivotal question during the Lie Theory meeting at the Tsinghua Sanya International Mathematics Forum (November 18, 2023): ``What about the generic parameter
$q$?". Indeed, the first author and collaborators constructed a novel one-parameter quantum affine algebra $
\mathcal U_q(\widehat{\mathfrak{sl}_2})$ of type
$A_1^{(1)}$ (see \cite{Feng Hu Zhuang JGP 2021, Hu Zhuang 2021}), termed the admissible quantum affine algebra, which is derived from the multi-parameter
setting (rather than the two-parameter setting discussed here) (see \cite{Pei Hu Rosso 2010}). This admissible object as Hopf algebra is proved to be not isomorphic to the standard quantum affine algebra $U_q(\widehat{\mathfrak{sl}_2})$ of Drinfeld-Jimbo type. This fully demonstrates that two- and multi-parameter quantum groups indeed offer us novel insights into Hopf algebra structure.

This paper is organized as follows.
In section 2, we recall some preliminaries needed in this paper: the unified definition of small two-parameters quantum groups, the Isomorphism Reduction Theorems for the small two-parameter quantum groups of types $A, B, C, D, F_4, G_2$,
and Radford's Theorem on simple Yetter-Drinfeld $H$-modules structures for a graded Hopf algebra $H$ with a finite abelian group algebra as its coradical.
Moreover, we also propose the Refined Isomorphism Reduction Theorem (Theorem 2.11), which will allow us to get neatly our classification results.
In section 3, we present our results obtained from the differences in their Hopf algebra structures for the above types.
Section 4 provides an alternative representation-theoretic approach to reach our classification results. Here, due to computational complexity, only a few illustrative examples in type $A_2$ with lower orders are provided.

It is believed that these novel and peculiar small quantum group structures, especially those without the Drinfeld double structure, may be appealing for the research of the recently popular non-semisimple topological quantum field theory (see Beliakova's recent work).

\section{Preliminaries}
\noindent{\it 2.1. \;}
  In \cite{Hu  Pei  Sci. in China Ser A: Math 2008},
the first author and Pei gave a unified definition of two-parameter quantum groups of any semisimple Lie algebras, using the Euler form.
  Let $C=(a_{ij})_{i, j\in I}$ be a Cartan matrix of finite type and $\mathfrak{g}$ the associated semisimple Lie algebra over $\mathbb{Q}$.
      Let $\{d_i\mid i\in I\}$ be a set of relatively prime positive integers such that $d_ia_{ij}=d_j a_{ji},\; i, j\in I$.
      Let $\langle\cdot, \cdot \rangle$ be the bilinear form, which is called the Euler form (or Ringel form),
      defined on the root lattice $Q$ by
      $$
      \langle i, j\rangle:=\langle \alpha_i, \alpha_j\rangle=\left\{
    \begin{array}{ll}                                     d_ia_{ij}, & \hbox{$i<j$,} \\
    d_i, & \hbox{$i=j$,} \\                               0, & \hbox{$i>j$.}                                     \end{array}                                       \right.
    $$

Thus the definition of small two-parameter quantum group is given as follows.

Assume that $r$ is a primitive $d$th root of unity,
$s$ is a primitive $d'$th root of unity, and $\ell$ is the least common multiple of $d$ and $d'$.
Set $r_i=r^{d_i}, s_i=s^{d_i}$.

\begin{defi} $($\cite{Bai Xiaotang 2008,BW1,R Chen PhD,Chen Hu and Wang to appear,Hu  Pei  Sci. in China Ser A: Math 2008,Hu Wang  Pacific J. Math 2009,Hu  Wang  JGP 2010}$)$
The small two-parameter quantum group $\mathfrak{u}_{r, s}(\mathfrak{g})$ is a unital associative algebra over an algebraically closed field  $\mathbb{K}\supset\mathbb{Q}(r, s)$,
generated by $e_i, f_i, \omega_i^\pm$, $\omega_i^{\prime\pm} \ (i\in I)$, subject to the following relations:

    \vspace{0.5em}

     $(R1)$ \ $\omega_i^{\pm 1}, \omega_j^{\prime \pm 1}$ commute with each other,
     $\omega_i\omega_i^{- 1}=1=\omega_j^\prime\omega_j^{\prime -1},\ \omega_i^\ell=\omega_i^{\prime \ell}=1$.

    \vspace{0.5em}
    $(R2)$ \ $\omega_ie_j\omega_i^{- 1}=r^{\langle j, i\rangle }s^{-\langle i, j\rangle} e_j,\quad
    \omega_i^\prime e_j\omega_i^{\prime -1}=r^{-\langle i, j\rangle}s^{\langle j, i\rangle} e_j$,

     \vspace{0.5em}
    $(R3)$ \ $\omega_i f_j\omega_i^{-1}=r^{-\langle j, i\rangle }s^{\langle i, j\rangle}f_j, \quad
    \omega_i^\prime f_j\omega_i^{\prime-1}=r^{\langle i, j\rangle }s^{-\langle j, i\rangle} f_j$,

      \vspace{0.5em}
    $(R4)$ \ $[e_i, f_i]=\delta_{ij}\dfrac{\omega_i-\omega_i^\prime}{r_i-s_i}$,

     \vspace{0.5em}
    $(R5)\ \sum\limits_{k=0}^{1-a_{ij}}(-1)^k\binom{1-a_{ij}}{k}_{r_is_i^{-1}}c_{ij}^{(k)}e_i^{1-a_{ij}-k}e_j e_i^k=0, \quad i\neq j$,

     \vspace{0.5em}
    $(R6)$ \ $\sum\limits_{k=0}^{1-a_{ij}}(-1)^k\binom{1-a_{ij}}{k}_{r_is_i^{-1}}c_{ij}^{(k)}f_i^k f_j f_i^{1-a_{ij}-k} =0,\quad i\neq j$,
    
   where $c_{ij}^{(k)}=(r_is_i^{-1})^{\frac{k(k-1)}{2}}r^{k\langle j, i\rangle} s^{-k\langle i, j\rangle}\ (i\neq j)$.
For a symbol $v$, we set the notations:
$$(n)_v=\dfrac{v^n-1}{v-1}, \quad (n)_v!=(1)_v(2)_v\cdots (n)_v,$$
$$\binom{n}{k}_v=\dfrac{(n)_v!}{(k)_v!(n-k)_v!}, \quad n\geq k\geq 0.$$ 

    \vspace{0.5em}
    $(R7)\;\;\; e_\alpha^\ell=f_\alpha^\ell=0$, where $\alpha$ is any positive root,
    where $e_\alpha, f_\alpha$ can be constructed in each type via good Lyndon words, for any positive root $\alpha$.

     \vspace{0.5em}
\noindent

The algebra $\mathfrak{u}_{r, s}(\mathfrak{g})$ has a Hopf structure with the comultiplication and the antipode given by
\begin{gather*}
\Delta(e_i)=e_i\otimes1+\omega_i\otimes e_i,\quad  \Delta(f_i)=1\otimes f_i+f_i\otimes \omega_i^\prime,\\
\varepsilon(\omega_i^{\pm1})=\varepsilon(\omega_i^\prime)^{\pm1}=1, \quad \varepsilon(e_i)=\varepsilon(f_i)=0,\\
S(\om_i^{\pm1})=\om_i^{\mp1}, \qquad
S({\om_i'}^{\pm1})={\om_i'}^{\mp1},\\
S(e_i)=-\om_i^{-1}e_i,\qquad S(f_i)=-f_i\,{\om_i'}^{-1}.
\end{gather*}
\end{defi}

\begin{prop} $($\cite{Bai Xiaotang 2008,BW1,R Chen PhD,Chen Hu and Wang to appear,Hu Wang  Pacific J. Math 2009,Hu  Wang  JGP 2010}$)$
  $\ \mathfrak{u}_{r, s}(\mathfrak{g})$ is a pointed Hopf algebra.
\end{prop}

\noindent{\it 2.2. \;}
We firstly recall the isomorphism reduction theorems for small two-parameter quantum groups of types $A, B, C, D, F_4, G_2$
(see Theorems \ref{thm isoA} to \ref{thm isoG2}).
For giving a neat classification, we need to improve and put forward the refined isomorphism reduction theorem (see Theorem \ref{thm refineA} and Corollaries \ref{cor isoA}, and \ref{cor2 isoA}).

\begin{theorem} $($\cite{Hu  Wang  JGP 2010}$)$ \label{thm isoA}
   Assume that $rs^{-1},\, r^\prime s^{\prime -1}$ are primitive $\ell$th roots of unity.
  Then $\varphi: \mathfrak{u}_{r,s}(\mathfrak{sl}_n)\cong \mathfrak{u}_{r^\prime, s^\prime}(\mathfrak{sl}_n)$
  as Hopf algebras if and only if either

 \vspace{0.5em}
  $(1)$ \ $(r^\prime, s^\prime)=(r, s), \; \varphi$ is a diagonal isomorphism, that is,
  $\varphi(\omega_i)=\widetilde{\omega}_i, \;\varphi(\omega_i^\prime)=\widetilde{\omega}_i^{\prime} ,\;
  \varphi(e_i)=a_i\tilde{e}_i,\; \varphi(f_i)=a_i^{-1}\tilde{f}_i;\;$ or

 \vspace{0.5em}
  $(2)$ \ $(r^\prime, s^\prime)=(s, r),\;
  \varphi(\omega_i)=\widetilde{\omega}_i^{\prime  -1},\;
  \varphi(\omega_i^\prime)=\widetilde{\omega}_i^{-1},\;
  \varphi(e_i)=a_i\tilde{f}_i\widetilde{\omega}_i^{\prime -1},\;
  \varphi(f_i)=a_i^{-1}\widetilde{\omega}_i^{-1}\tilde{e}_i;\;$ or

 \vspace{0.5em}
  $(3)$ \ $(r^\prime, s^\prime)=(s^{-1}, r^{-1}),\;
  \varphi(\omega_i)=\widetilde{\omega}_{n-i},\; \varphi(\omega_i^\prime)=\widetilde{\omega}_{n-i}^{\prime},\;
  \varphi(e_i)=a_{n-i}\tilde{e}_{n-i}, \;
  $ $\varphi(f_i)=(rs)^{-1}a_{n-i}^{-1}\tilde{f}_{n-i};$\; or

 \vspace{0.5em}
   $(4)$ \ $(r^\prime, s^\prime)=(r^{-1}, s^{-1}),\;
   \varphi(\omega_i)=\widetilde{\omega}_{n-i}^{\prime  -1},\;
  \varphi(\omega_i^\prime)=\widetilde{\omega}_{n-i}^{-1},\;
  \varphi(e_i)=a_{n-i}\tilde{f}_{n-i}\widetilde{\omega}_{n-i}^{\prime -1}$,
  $\varphi(f_i)=(rs)^{-1}a_{n-i}^{-1}\widetilde{\omega}_{n-i}^{-1}\tilde{e}_{n-i} \ (a_i\in \Bbb K^\ast).$
\end{theorem}

\begin{theorem} $($\cite{Hu  Wang  JGP 2010}$)$
  Assume that  $rs^{-1}, r^\prime s^{\prime -1}$ are primitive $\ell$th roots of unity with $\ell \neq 3,4$,
  and $\zeta$ is the $2$nd root of unity.
  Then $\varphi: \mathfrak{u}_{r,s}(\mathfrak{so}_{2n+1})\cong \mathfrak{u}_{r^\prime, s^\prime}(\mathfrak{so}_{2n+1})$
  as Hopf algebras if and only if either

 \vspace{0.5em}
  $(1)$ \ $(r^\prime, s^\prime)=\zeta(r, s), \ \varphi$ is a diagonal isomorphism

   \vspace{0.5em}
  $\varphi(\omega_i)=\widetilde{\omega}_i, \;\varphi(\omega_i^\prime)=\widetilde{\omega}_i^{\prime} ,\;
  \varphi(e_i)=a_i\tilde{e}_i, \;\varphi(f_i)=\zeta^{\delta_{i, n}} a_i^{-1}\tilde{f}_i;\;$ or

\vspace{0.5em}
$(2)$ \ $(r^\prime, s^\prime)=\zeta(s, r),\;
  \varphi(\omega_i)=\widetilde{\omega}_i^{\prime  -1},\;
  \varphi(\omega_i^\prime)=\widetilde{\omega}_i^{-1},\;
  \varphi(e_i)=a_i\tilde{f}_i\widetilde{\omega}_i^{\prime -1}, $

   \vspace{0.5em}
  $\varphi(f_i)=\zeta^{\delta_{i, n}}a_i^{-1}\widetilde{\omega}_i^{-1}\tilde{e}_i, \ (a_i\in \mathbb{K}^\ast).$
\end{theorem}

\begin{theorem} $($\cite{R Chen PhD}$)$
  Assume that $n\geq 2$, and  $rs^{-1}\text{and}\; r^\prime(s^\prime)^{-1}$ are primitive $\ell$th roots of unity with $\ell \neq 2, 3$.
Then $\varphi: \mathfrak{u}_{r,s}(\mathfrak{sp}_{2n}) \cong \mathfrak{u}_{r^\prime, s^\prime}(\mathfrak{sp}_{2n})$
as Hopf algebras if and only if either
 $(r^\prime, s^\prime)=(r, s)$, or $(r^\prime, s^\prime)=(-r, -s)$.
\end{theorem}

\begin{theorem} $($\cite{Bai Xiaotang 2008}$)$
  Assume that $n\geq 4,\; rs^{-1} \text{and }r^\prime(s^\prime)^{-1}$ are primitive $\ell$th roots of unity with $\ell \neq 2$.
Then $\varphi: \mathfrak{u}_{r,s}(\mathfrak{so}_{2n}) \cong \mathfrak{u}_{r^\prime, s^\prime}(\mathfrak{so}_{2n})$
as Hopf algebras if and only if either
 $(r^\prime, s^\prime)=(r, s)$ or $(r^\prime, s^\prime)=(s^{-1}, r^{-1})$.
\end{theorem}

\begin{theorem} $($\cite{Chen Hu and Wang to appear}$)$
   Assume that  $rs^{-1}, r^\prime s^{\prime -1}$ are primitive $\ell$th roots of unity,
  and $\ell\neq 3, 4,\; \zeta$ is the $2$nd root of unity.
  Then $\varphi: \mathfrak{u}_{r,s}(F_4)\cong \mathfrak{u}_{r^\prime, s^\prime}(F_4)$ as Hopf algebras if and only if either

   \vspace{0.5em}
  $(1)$ \ $\;(r^\prime, s^\prime)=\zeta(r, s), \  \varphi$ is diagonal isomorphism

  \vspace{0.5em}
  $\varphi(\omega_i)=\widetilde{\omega}_i, \;\varphi(\omega_i^\prime)=\widetilde{\omega}_i^{\prime},\;
  \varphi(e_i)=a_i\widetilde{e}_i,\;
\varphi(f_i)=\zeta^{\delta_{i, 3}+\delta_{i, 4}} a_i^{-1}\widetilde{f}_i;\;$ or

 \vspace{0.5em}
$(2)$ \ $\;(r^\prime, s^\prime)=\zeta(s, r),\;
  \varphi(\omega_i)=\widetilde{\omega}_i^{\prime  -1},\;
  \varphi(\omega_i^\prime)=\widetilde{\omega}_i^{-1},\;
  \varphi(e_i)=a_i\widetilde{f}_i\widetilde{\omega}_i^{\prime -1},\;$

  \vspace{0.5em}
  $\varphi(f_i)=\zeta^{\delta_{i, 3}+\delta_{i, 4}}a_i^{-1}\widetilde{\omega}_i^{ -1}\widetilde{e}_i, \ (a_i\in \mathbb{K}^\ast).$
\end{theorem}

\begin{theorem} $($\cite{Hu Wang  Pacific J. Math 2009}$)$ \label{thm isoG2}
  Assume that  $rs^{-1}, r^\prime s^{\prime -1}$ are primitive $\ell$th roots of unity,
  and $\ell\neq  4, 6, \;\zeta$ is a $3$rd root of unity.
  Then $\varphi: \mathfrak{u}_{r,s}(G_2)\cong \mathfrak{u}_{r^\prime, s^\prime}(G_2)$ as Hopf algebras if and only if either

   \vspace{0.5em}

  $(1)$ \ $\;(r^\prime, s^\prime)=\zeta(r, s),\  \varphi$ is diagonal isomorphism

   \vspace{0.5em}
  $\varphi(\omega_i)=\widetilde{\omega}_i, \;\varphi(\omega_i^\prime)=\widetilde{\omega}_i^{\prime} ,\;
  \varphi(e_i)=a_i\tilde{e}_i,\;
 \varphi(f_i)=\zeta^{\delta_{i, 1}} a_i^{-1}\tilde{f}_i;\;$ or

  \vspace{0.5em}

$(2)$ \ $\;(r^\prime, s^\prime)=\zeta(s, r),\;
  \varphi(\omega_i)=\widetilde{\omega}_i^{\prime  -1},\;
  \varphi(\omega_i^\prime)=\widetilde{\omega}_i^{-1},\;
  \varphi(e_i)=a_i\tilde{f}_i\widetilde{\omega}_i^{\prime -1},\;$

    \vspace{0.5em}
  $\varphi(f_i)=\zeta^{\delta_{i, 1}}a_i^{-1}\widetilde{\omega}_i^{ -1}\tilde{e}_i, \ (a_i\in \mathbb{K}^\ast).$
\end{theorem}

\begin{remark}\label{rmk refine}
  For classification purpose, we will provide a refined reduction isomorphism theorem (see Theorem \ref{thm refineA}) .
  The condition that ``$rs^{-1}$ and $r's'^{-1}$ are primitive $\ell$th roots of unity'' is not essential and can be further improved.
  Indeed, by the formula below ,
  \begin{equation}\label{e_i}
    \Delta(e_j^a)=\sum_{i=0}^{a}s_j^{i(i-a)}\left[a\atop i\right]_j
e_{j}^i\omega_{j}^{a-i} \otimes e_{j}^{a-i},
  \end{equation}
\begin{equation}\label{f_i}
  \Delta(f_j^a)=\sum_{i=0}^{a}r_j^{i(i-a)}\left[a\atop i\right]_j
f_{j}^{a-i} \otimes\omega_{j}'^{a-i}f_{j}^i,
\end{equation}
where $[c]_j:=\frac{r_j^c-s_j^c}{r_j-s_j},\ [c]_j!:=[c]_j[c-1]_j\cdots
[2]_j[1]_j,\ \left[c\atop d\right]_j:=\frac{[c]_j!}{[d]_j![c-d]_j!}$,
one easily sees that the following facts holds.

 Assume $rs^{-1}$ is a primitive $m$th root of unity,
 then we have $[m]_i=0$, and $[1]_i, [2]_i,\cdots, $ $ [m-1]_i$ are non-zero.
 (\ref{e_i}) and (\ref{f_i}) show that $e_i^m$ and $f_i^m$, which can also be multiplied by some group-like elements, are also skew-primitive elements.
 In particular, when $rs^{-1}$ is a primitive $\ell$th root of unity,
 the group-like elements together with the skew-primitive elements of $\mathfrak{u}_{r, s}(\mathfrak{g})$ take the form $\mathfrak{u}_1$ of the coradical filtration $\{u_k\}$ of $\mathfrak{u}_{r, s}(\mathfrak{g})$ (for definition, see \cite{Benkart Witherspoon Fields Institute Communications 2004, CM})
  $$\mathfrak{u}_1:=\mathbb{K} G+\sum_{i=1}^{n}\mathbb{K}e_i G+\mathbb{K} f_i G,$$
  where $G$ is the group generated by the group-like elements of $\mathfrak{u}_{r, s}(\mathfrak{g})$,
  and $n=\text{rank}(\mathfrak{g})$.
\end{remark}

Recall that given two group-like elements $g, h$ in a Hopf algebra $H$, let
$$
P_{g,h}(H)=\{x\in H\mid \Delta(x)=x\otimes g+h\otimes x\}.
$$
Denote the set of  $(g, h)$ skew-primitive elements.

\begin{lemm}\label{lemm skewp}
Assume that $r$ is a primitive $d$th root of unity,
$s$ is a primitive $d'$th root of unity, and $\ell$ is the least common multiple of $d$ and $d'$.
Further, $rs^{-1}$ is a primitive $m$th root of unity, where $m\neq \ell$.
Then

$({\textrm{\rm i}})$ \;\
$P_{1,\omega_i}(\mathfrak{u}_{r,s})=\mathbb{K}(1-\omega_i)+\mathbb{K}e_i;$
$\quad
P_{1,\omega_i'^{-1}}(\mathfrak{u}_{r,s})=\mathbb{K}(1-\omega_i'^{-1})+\mathbb{K}f_i\omega_i'^{-1};$

$\qquad P_{1,\omega_i^m}(\mathfrak{u}_{r,s})=\mathbb{K}(1-\omega_i^m)+\mathbb{K}e_i^m;$
$\quad
P_{1,\omega_i'^{-m}}(\mathfrak{u}_{r,s})=\mathbb{K}(1-\omega_i'^{-m})+\mathbb{K}f_i^m\omega_i'^{-m};$

$\qquad P_{1,\sigma}(\mathfrak{u}_{r,s})=\mathbb{K}(1-\sigma), \quad
\textrm{for}\
 \sigma\not\in\{\omega_i, \, \omega_i'^{-1}, \omega_i^m,  \omega_i'^{-m} \mid 1\le i\le n\}.$

\smallskip
$({\rm ii})$\;
$P_{\omega'_i,1}(\mathfrak{u}_{r,s})=\mathbb{K}(1-\omega'_i)+\mathbb{K}f_i;$
$\quad
P_{\omega_i^{-1},1}(\mathfrak{u}_{r,s})=\mathbb{K}(1-\omega_i^{-1})+\mathbb{K}e_i\omega_i^{-1};$

$\qquad P_{\omega_i'^m, 1}(\mathfrak{u}_{r,s})=\mathbb{K}(1-\omega_i'^m)+\mathbb{K}f_i^m;$
$\quad
P_{\omega_i^{-m}, 1}(\mathfrak{u}_{r,s})=\mathbb{K}(1-\omega_i^{-m})+\mathbb{K}e_i^m\omega_i^{-m};$

$\qquad P_{\sigma,1}(\mathfrak{u}_{r,s})=\mathbb{K}(1-\sigma), \quad
\textrm{for}\
 \sigma\not\in\{\omega_i^{-1}, \omega'_i,\omega_i'^m, \omega_i^{-m}  \mid 1\le i\le n\}.$
\end{lemm}

\begin{proof}
  According to Remark \ref{rmk refine}, we can prove this Lemma similarly to the proof of Lemma 4.3 in \cite{Hu Wang  Pacific J. Math 2009}.
\end{proof}

Therefore, we can generalize the original isomorphism reduction theorems (Theorems \ref{thm isoA}--\ref{thm isoG2})
under the assumption that requires both $rs^{-1}$ and $r^\prime s^{\prime -1}$ to be  primitive $m$th (rather than $\ell$th) roots of unity. Certainly, $m\,|\,\ell$ in our context.

\begin{theorem} \label{thm refineA} $($Refined Isomorphism Reduction Theorem$)$
Assume $\emph{lcm}(|r|, |s|)$ $=\emph{lcm}(|r’|, |s'|)=\ell$, where $|r|$ denotes the order of $r$.
Under the assumption that both $rs^{-1}$ and $r^\prime s^{\prime -1}$ are primitive $m$th roots of unity,
the original isomorphism reduction theorems still hold.
\end{theorem}

\begin{proof}
  Assume there exists an isomorphism $\varphi: \mathfrak{u}_{r, s}(\mathfrak{g})\cong \mathfrak{u}_{r^\prime, s^\prime}(\mathfrak{g})$,
  and $\text{rank}(\mathfrak{g})\\=n$.
  We denote the generators of $\mathfrak{u}_{r',s'}(\mathfrak{g})$ by $\tilde{e}_i$, $\tilde{f}_i,\, \tilde\omega_i$, $\tilde\omega'_i, 1\le i\le n$.
Since $\varphi$ is a Hopf algebra isomorphism, we have
$$\Delta(\varphi(e_i))=(\varphi\otimes\varphi)
(\Delta(e_i))=\varphi(e_i)\otimes 1+\varphi(\omega_i)\otimes
\varphi(e_i),
$$
$\varphi(e_i) \in P_{1,\varphi(\omega_i)}(\mathfrak{u}_{r',s'})$,
and
$\varphi(\omega_i)\in\mathbb{K}\tilde{G}$.

We conclude that $\varphi(\omega_i)\in\{\,\tilde\omega_j, \tilde\omega_j'^{-1}, \tilde\omega_j^m, \tilde\omega_j'^{-m}  \mid
1\le j\le n\,\}$.
However, the latter two cases are impossible because $\varphi$ must preserve the order of $\omega_i$ and obviously $\text{gcd}(m, \ell)=m\neq 1$.
We thus have either $\varphi(\omega_i)=\tilde\omega_j$,
$\varphi(e_i)=a(1-\tilde\omega_j)+b\tilde e_j$,
or $\varphi(\omega_i)=\tilde\omega_j'^{-1}$,
$\varphi(e_j)=a(1-\tilde\omega_j'^{-1})+b\tilde f_j\tilde\omega_j'^{-1}$ for some $1\le j\le n$ and $a,b\in\mathbb K$.
Then we can complete the proof similarly as those of Theorems \ref{thm isoA}--\ref{thm isoG2}.
\end{proof}

\begin{coro}\label{cor isoA}
  Assume that $\emph{lcm}(|r|, |s|)=\emph{lcm}(|r’|, |s'|)=\ell$,
  $rs^{-1}$ is a primitive $\ell$th root of unity,
  and $r's'^{-1}$ is a primitive $m$th root of unity with $m<\ell$.
  Then it necessarily holds that
  $$\mathfrak{u}_{r, s}(\mathfrak{g}) \ncong \mathfrak{u}_{r', s'}(\mathfrak{g}).$$
\end{coro}

\begin{proof}
  If there is such an isomorphism $\varphi$, then $\varphi$ must map skew-primitive elements of $\mathfrak{u}_{r, s}(\mathfrak{g})$ to those of $\mathfrak{u}_{r', s'}(\mathfrak{g})$, and vice versa.
  This obviously yields a contradiction, since $\mathfrak{u}_{r', s'}(\mathfrak{g})$ has additional skew-primitive elements such as $e_i^m$, $f_i^m$.
\end{proof}

\begin{coro}\label{cor2 isoA}
Assume that $\emph{lcm}(|r|, |s|)=\emph{lcm}(|r’|, |s'|)=\ell$,
and that $rs^{-1}$ is a primitive $m$th root of unity,
while $r^\prime s^{\prime -1}$ is a  primitive $m'$th root of unity such that  $1<m\neq  m'\leq \ell $.
Then it necessarily holds that
 $$\mathfrak{u}_{r, s}(\mathfrak{g}) \ncong \mathfrak{u}_{r', s'}(\mathfrak{g}).$$
That is to say, the order of $rs^{-1}$ is an invariant when investigating the isoclasses of $\mathfrak{u}_{r, s}(\mathfrak{g})$.
\end{coro}

\begin{proof}
   According to Corollary \ref{cor isoA}, we may assume $m< \ell$.
   If there exists such an isomorphism $\varphi$, then
    $\varphi(e_i^m)\in P_{1, \varphi(\omega_i^m)}$.
    By Lemma \ref{lemm skewp}, $\varphi(\omega_i^m)\in \{\tilde\omega_i, \tilde\omega_i'^{-1}, \tilde\omega_i^{m'},  \tilde\omega_i'^{-m'}\}$.
   By comparing the orders of both sides, we conclude that this is  impossible.
\end{proof}

\noindent{\it 2.3. \;}
The dimension of $\mathfrak{u}_{r, s}(\mathfrak{g})$ can be calculated.
For further use, we list the dimension of $\mathfrak{u}_{r, s}(\mathfrak{sl}_n)$ here.

\begin{prop} $($\cite{Benkart Witherspoon Fields Institute Communications 2004}\label{prop dimension of algebra}$)$
  Assume $r$ is a primitive $d$th root of unity,
$s$ is a primitive $d'$th root of unity,
and $\ell$ is the least common multiple of $d$ and $d'$.
Then $\mathfrak{u}_{r, s}(\mathfrak{sl}_n)$ is an algebra of dimension $\ell^{(n+2)(n-1)}$
equipped with a monomial convex PBW Lyndon basis. The basis elements are of the form
$$
\mathcal{E}_{i_1, j_1}^{a_1}\cdots \mathcal{E}_{i_p, j_p}^{a_p}\omega_1^{b_1}\cdots \omega_{n-1}^{b_{n-1}}
(\omega'_1)^{b'_1}\cdots (\omega'_{n-1})^{b'_{n-1}}
\mathcal{F}_{i'_1, j'_1}^{a_1'}\cdots \mathcal{F}_{i_p', j_p'}^{a_p'},
$$
where $(i_1, j_1)<\cdots<(i_p, j_p)$ and $(i_1', j_1')<\cdots< (i_p', j_p')$ lexicographically,
and all powers range between $0$ and $\ell-1$.
\end{prop}

\noindent{\it 2.4. \;}
Radford gave a description of all simple Yetter-Drinfeld modules for certain graded Hopf algebras:
\begin{lemm} $($\cite{Radford construction J Algebra 2003}$)$
  Let $H$ be a bialgebra over the field $\Bbbk$ and suppose that $H^{op}$ is a Hopf algebra with antipode $S^ {op}$.
  Let $\beta\in Alg_{\Bbbk}(H, \Bbbk)$,
   then $H_\beta=(H, \bullet_\beta, \Delta)\in{ _H \mathcal{Y}\mathcal{D}^H}$, where
   $$x\bullet_\beta a=\sum \beta(x_{(2)}) x_{(3)}a S^{op}(x_{(1)}),$$
   for all $x,a \in H$.

\end{lemm}

\begin{theorem} \label{thm2.11} $($\cite{Radford construction J Algebra 2003}$)$
   Let $H=\bigoplus_{n=0}^\infty H_n$ be a graded Hopf algebra over an algebraically closed field of characteristic $0$.
 Suppose that $H_0=\Bbbk G$ for some finite abelian group, and $H_n=H_{n+1}=\cdots =(0)$ for some $n>0$.
 Then
 $$(\beta, g)\longmapsto [H\bullet_\beta g]$$
 is a bijective correspondence between the Cartesian product of sets $\widehat{G}\times G$ and the set of isoclasses of simple Yetter-Drinfeld $H$-modules.
\end{theorem}

\begin{remark}  As we have known,
 $\mathfrak{u}_{r, s}(\mathfrak{g})$ has Drinfeld double structure $\mathcal{D}(H)$ when parameters $r,\, s$ subject to some conditions \cite{Bai Xiaotang 2008, BW1, R Chen PhD, Chen Hu and Wang to appear, Hu Wang  Pacific J. Math 2009, Hu  Wang  JGP 2010}.
 In this case, the module category is equivalent to the Yetter-Drinfeld $H$-module category.
 What's more, Pereira gave a solution \cite{Benkart Pereira Witherspoon 2010} to calculate the dimension of $H\bullet_\beta g$, using the computer algebra system S\begin{footnotesize}INGULAR\,\end{footnotesize}:\,:\,P\begin{footnotesize}LURAL\end{footnotesize} \cite{Singular}.
\end{remark}

\section {New Exotic Isoclasses via the Refined Isomorphism Reduction Theorem and the Convex PBW Lyndon Basis Theorem}
\noindent{\it 3.1. Type $A$\;}
(1) Assume that $q$ is a primitive $4$th root of unity.
\begin{prop}\label{prop A 4}
  Assume that $r^4=s^4=1$.
Then Table \ref{table A 4} provides a complete classification of $\mathfrak{u}_{r, s}(\mathfrak{sl}_n)$.
  \begin{table}[h]
  \centering
  \caption{Isoclasses of type $A_{n-1}$ small quantum groups when $q$ is a primitive $4$th root of unity}\label{table A 4}
\begin{tabular}{|c|l|c|}
  \hline
  & Isoclasses & Dimension \\ \hline
  $1$ & $\mathfrak{u}_{1, q}(\mathfrak{sl}_n)\cong \mathfrak{u}_{q, 1}(\mathfrak{sl}_n)\cong \mathfrak{u}_{q^3, 1}(\mathfrak{sl}_n)\cong \mathfrak{u}_{1, q^3}(\mathfrak{sl}_n);$  & $4^{(n+2)(n-1)}$ \\ \hline
  $2$ & $\mathfrak{u}_{q, q^2}(\mathfrak{sl}_n)\cong \mathfrak{u}_{q^2, q}(\mathfrak{sl}_n)\cong \mathfrak{u}_{q^3, q^2}(\mathfrak{sl}_n)\cong \mathfrak{u}_{q^2, q^3}(\mathfrak{sl}_n);$ & $4^{(n+2)(n-1)}$ \\ \hline
  $3$ & $\mathfrak{u}_{1, q^2}(\mathfrak{sl}_n)\cong \mathfrak{u}_{q^2, 1}(\mathfrak{sl}_n);$ & $2^{(n+2)(n-1)}$\\ \hline
  $4$ & $\mathfrak{u}_{q, q^3}(\mathfrak{sl}_n)=\mathfrak{u}_{q, q^{-1}}(\mathfrak{sl}_n)\cong\mathfrak{u}_{q^{-1},q}(\mathfrak{sl}_n)= \mathfrak{u}_{q^3, q}(\mathfrak{sl}_n).$ & $4^{(n+2)(n-1)}$ \\  \hline
\end{tabular}
\end{table}
\end{prop}

\begin{proof}
  Firstly, using Theorem \ref{thm refineA}, we know that the candidates $1$ and $2$ are non-isomorphic,
  and candidates $3$ and $4$ are also non-isomorphic.
  Secondly, the candidate $3$ is not isomorphic to the others since its dimension is distinct from the others by Proposition \ref{prop dimension of algebra}.
  Lastly, we conclude that the candidates $1$, $2$, and $4$ are pairwise non-isomorphic by Corollary \ref{cor isoA}.
\end{proof}

(2) Assume that $q$ is a primitive $5$th root of unity.
It implies that $rs^{-1}$ is a primitive $5$th root of unity for all of these candidates.
According to the Isomorphism Reduction Theorem of Type $A$, we have the following
\begin{prop}\label{thm A 5}
  Assume that $r^5=s^5=1$.
  Then Table \ref{table A 5} provides the complete classification of $\mathfrak{u}_{r, s}(\mathfrak{sl}_n)$.
\begin{table}[h]
  \centering
  \caption{Isoclasses of type $A$ small quantum groups when $q$ is a primitive $5$th root of unity}\label{table A 5}
\begin{tabular}{|c|l|}
  \hline
   1 & $\mathfrak{u}_{1, q}(\mathfrak{sl}_n)\cong \mathfrak{u}_{q, 1}(\mathfrak{sl}_n)\cong \mathfrak{u}_{q^4, 1}(\mathfrak{sl}_n)\cong \mathfrak{u}_{1, q^4}(\mathfrak{sl}_n);$ \\ \hline
  2 & $\mathfrak{u}_{q, q^2}(\mathfrak{sl}_n)\cong \mathfrak{u}_{q^2, q}(\mathfrak{sl}_n)\cong \mathfrak{u}_{q^3, q^4}(\mathfrak{sl}_n)\cong \mathfrak{u}_{q^4, q^3}(\mathfrak{sl}_n);$ \\ \hline
  3 & $\mathfrak{u}_{q^2, q^3}(\mathfrak{sl}_n)\cong\mathfrak{u}_{q^3, q^2}(\mathfrak{sl}_n);$\\ \hline
  4 & $\mathfrak{u}_{1, q^2}(\mathfrak{sl}_n)\cong \mathfrak{u}_{q^2, 1}(\mathfrak{sl}_n)
\cong \mathfrak{u}_{1, q^3}(\mathfrak{sl}_n)\cong \mathfrak{u}_{q^3, 1}(\mathfrak{sl}_n);$ \\  \hline
  5 & $\mathfrak{u}_{q, q^3}(\mathfrak{sl}_n)\cong \mathfrak{u}_{q^3, q}(\mathfrak{sl}_n)
\cong \mathfrak{u}_{q^2, q^4}(\mathfrak{sl}_n)\cong \mathfrak{u}_{q^4, q^2}(\mathfrak{sl}_n);$ \\  \hline
  6 & $\mathfrak{u}_{q^{-1},q}(\mathfrak{sl}_n)=\mathfrak{u}_{q^4, q}(\mathfrak{sl}_n)\cong\mathfrak{u}_{q, q^4}(\mathfrak{sl}_n)=\mathfrak{u}_{q, q^{-1}}(\mathfrak{sl}_n).$\\  \hline
\end{tabular}
\end{table}
\end{prop}

(3) More generally, assume that $q$ is a primitive root of unity with an odd prime order $p$.
Using the Isomorphism Reduction Theorem, we have the following

\begin{theorem}\label{thm A general}
  Assume that $r^p=s^p=1$, where $p$ is an odd prime.
Then Table \ref{table A general} lists the complete classification of $\mathfrak{u}_{r, s}(\mathfrak{sl}_n)$.
  Totally, there are $\frac{p^2-1}{4}$ isoclasses.
  \begin{table}
 \centering
  \caption{Isoclasses of type $A$ small quantum groups when $q$ is a primitive root of unity with odd prime order $p$}\label{table A general}
  \begin{tabular}{|l|}
    \hline
     $\mathfrak{u}_{q^k, q^{k+t}}(\mathfrak{sl}_n)\cong \mathfrak{u}_{q^{k+t}, q^k}(\mathfrak{sl}_n)
     \cong \mathfrak{u}_{q^{p-k}, q^{p-k-t}}(\mathfrak{sl}_n)\cong \mathfrak{u}_{q^{p-k-t}, q^{p-k}}(\mathfrak{sl}_n);$ \\
     \\
      $(1\leq t\leq \frac{p-1}{2} , 0\leq k \leq p-1,\; t, k\in \mathbb{Z})$ \\ \hline
  \end{tabular}
  \end{table}
\end{theorem}

(4) Assume that $q$ is a primitive $6$th root of unity.
\begin{prop}\label{prop A 6}
Assume that $r^6=s^6=1$. Then Table \ref{table A 6} provides the complete classification of $\mathfrak{u}_{r, s}(\mathfrak{sl}_n)$.
\begin{table}[h]
  \centering
   \renewcommand{\tablename}{Table}
  \caption{Isoclasses of type $A_{n-1}$ small quantum groups when $q$ is a primitive $6$th root of unity}\label{table A 6}
\begin{tabular}{|c|l|c|}
  \hline
  & Isoclasses & Dimension \\ \hline
  $1$ & $\mathfrak{u}_{1, q}(\mathfrak{sl}_n)\cong \mathfrak{u}_{q, 1}(\mathfrak{sl}_n)\cong \mathfrak{u}_{q^5, 1}(\mathfrak{sl}_n)\cong \mathfrak{u}_{1, q^5}(\mathfrak{sl}_n);$ & $6^{(n+2)(n-1)}$ \\ \hline
  $2$ & $\mathfrak{u}_{q, q^2}(\mathfrak{sl}_n)\cong \mathfrak{u}_{q^2, q}(\mathfrak{sl}_n)\cong \mathfrak{u}_{q^4, q^5}(\mathfrak{sl}_n)\cong \mathfrak{u}_{q^5, q^4}(\mathfrak{sl}_n);$ & $6^{(n+2)(n-1)}$ \\ \hline
  $3$ & $\mathfrak{u}_{q^2, q^3}(\mathfrak{sl}_n)\cong \mathfrak{u}_{q^3, q^2}(\mathfrak{sl}_n)
\cong \mathfrak{u}_{q^4, q^3}(\mathfrak{sl}_n)\cong \mathfrak{u}_{q^3, q^4}(\mathfrak{sl}_n);$ & $6^{(n+2)(n-1)}$\\ \hline
  $4$ & $\mathfrak{u}_{1, q^2}(\mathfrak{sl}_n)\cong \mathfrak{u}_{q^2, 1}(\mathfrak{sl}_n)
\cong \mathfrak{u}_{1, q^4}(\mathfrak{sl}_n)\cong \mathfrak{u}_{q^4, 1}(\mathfrak{sl}_n);$ &$3^{(n+2)(n-1)}$ \\  \hline
  $5$ & $\mathfrak{u}_{q, q^3}(\mathfrak{sl}_n)\cong \mathfrak{u}_{q^3, q}(\mathfrak{sl}_n)
\cong \mathfrak{u}_{q^5, q^3}(\mathfrak{sl}_n)\cong \mathfrak{u}_{q^3, q^5}(\mathfrak{sl}_n);$ & $6^{(n+2)(n-1)}$  \\  \hline
  $6$ & $\mathfrak{u}_{q, q^5}(\mathfrak{sl}_n)\cong \mathfrak{u}_{q^5, q}(\mathfrak{sl}_n);$& $6^{(n+2)(n-1)}$ \\  \hline
  $7$ & $\mathfrak{u}_{q^2, q^4}(\mathfrak{sl}_n)\cong \mathfrak{u}_{q^4, q^2}(\mathfrak{sl}_n);$ & $3^{(n+2)(n-1)}$  \\ \hline
  $8$ & $\mathfrak{u}_{1, q^3}(\mathfrak{sl}_n)\cong \mathfrak{u}_{q^3, 1}(\mathfrak{sl}_n);$ & $2^{(n+2)(n-1)}$ \\  \hline
  $9$ & $\mathfrak{u}_{q^2, q^5}(\mathfrak{sl}_n)\cong \mathfrak{u}_{q^5, q^2}(\mathfrak{sl}_n)
\cong \mathfrak{u}_{q, q^4}(\mathfrak{sl}_n)\cong \mathfrak{u}_{q^4, q}(\mathfrak{sl}_n).$ & $6^{(n+2)(n-1)}$ \\ \hline
\end{tabular}
\end{table}
\end{prop}

\begin{proof}
   Firstly, using Theorem \ref{thm refineA}, we conclude that candidates $1$--$3$ are  pairwise non-isomorphic,
   candidate $4$ is non-isomorphic to candidate $7$,
    and candidate $5$ is non-isomorphic to candidate $6$.
    Secondly, candidates $4$, $7$, and $8$ are non-isomorphic to the others since their dimensions are distinguished by Proposition \ref{prop dimension of algebra}.
    Finally, we claim that candidates $1$--$3, 5, 6$, and $9$ are pairwise non-isomorphic by Corollaries \ref{cor isoA} and \ref{cor2 isoA}.
\end{proof}

\begin{remark}
  If $q$ is a primitive $6$th root of unity, then $q^2$ is a primitive $3$rd root of unity.
  That is to say, Table \ref{table A 6} has contained the situation that $r^3=s^3=1$ (the candidates $4$ and $7$).
\end{remark}

(5) Assume that $q$ is an primitive $8$th root of unity.
\begin{prop}\label{prop A 8}
Assume that $r^8=s^8=1$.
Then Table \ref{table A 8} provides a complete classification of $\mathfrak{u}_{r, s}(\mathfrak{sl}_n)$.
\begin{table}[h]
  \centering
  \renewcommand{\tablename}{Table}
   \caption{Isoclasses of type $A_{n-1}$ small quantum groups when $q$ is a primitive $8$th root of unity}
 \label{table A 8}
\begin{tabular}{|c|l|c|}
  \hline
  & Isoclasses & Dimension \\ \hline
   $1$ & $\mathfrak{u}_{1, q}(\mathfrak{sl}_n)\cong \mathfrak{u}_{q, 1}(\mathfrak{sl}_n)\cong \mathfrak{u}_{q^7, 1}(\mathfrak{sl}_n)\cong \mathfrak{u}_{1, q^7}(\mathfrak{sl}_n);$ & $8^{(n+2)(n-1)}$ \\ \hline
  $2$ & $\mathfrak{u}_{q, q^2}(\mathfrak{sl}_n)\cong \mathfrak{u}_{q^2, q}(\mathfrak{sl}_n)\cong \mathfrak{u}_{q^6, q^7}(\mathfrak{sl}_n)\cong \mathfrak{u}_{q^7, q^6}(\mathfrak{sl}_n);$  & $8^{(n+2)(n-1)}$\\ \hline
  $3$ & $\mathfrak{u}_{q^2, q^3}(\mathfrak{sl}_n)\cong \mathfrak{u}_{q^3, q^2}(\mathfrak{sl}_n)
\cong \mathfrak{u}_{q^6, q^5}(\mathfrak{sl}_n)\cong \mathfrak{u}_{q^5, q^6}(\mathfrak{sl}_n);$& $8^{(n+2)(n-1)}$\\ \hline
  $4$ & $\mathfrak{u}_{q^3, q^4}(\mathfrak{sl}_n)\cong \mathfrak{u}_{q^4, q^3}(\mathfrak{sl}_n)
\cong \mathfrak{u}_{q^4, q^5}(\mathfrak{sl}_n)\cong \mathfrak{u}_{q^5, q^4}(\mathfrak{sl}_n);$& $8^{(n+2)(n-1)}$\\  \hline
  $5$ & $\mathfrak{u}_{1, q^3}(\mathfrak{sl}_n)\cong \mathfrak{u}_{q^3, 1}(\mathfrak{sl}_n)
\cong \mathfrak{u}_{q^5, 1}(\mathfrak{sl}_n)\cong \mathfrak{u}_{1, q^5}(\mathfrak{sl}_n);$ & $8^{(n+2)(n-1)}$\\  \hline
  $6$ & $\mathfrak{u}_{q, q^4}(\mathfrak{sl}_n)\cong \mathfrak{u}_{q^4, q}(\mathfrak{sl}_n)
\cong \mathfrak{u}_{q^7, q^4}(\mathfrak{sl}_n)\cong \mathfrak{u}_{q^4, q^7}(\mathfrak{sl}_n);$& $8^{(n+2)(n-1)}$ \\  \hline
  $7$ & $\mathfrak{u}_{q^2, q^5}(\mathfrak{sl}_n)\cong \mathfrak{u}_{q^5, q^2}(\mathfrak{sl}_n)
\cong \mathfrak{u}_{q^3, q^6}(\mathfrak{sl}_n)\cong \mathfrak{u}_{q^6, q^3}(\mathfrak{sl}_n);$& $8^{(n+2)(n-1)}$\\  \hline
  $8$ & $\mathfrak{u}_{q, q^6}(\mathfrak{sl}_n)\cong \mathfrak{u}_{q^6, q}(\mathfrak{sl}_n)
\cong \mathfrak{u}_{q^2, q^7}(\mathfrak{sl}_n)\cong \mathfrak{u}_{q^7, q^2}(\mathfrak{sl}_n);$& $8^{(n+2)(n-1)}$\\  \hline
  $9$ & $\mathfrak{u}_{1, q^2}(\mathfrak{sl}_n)\cong \mathfrak{u}_{q^2, 1}(\mathfrak{sl}_n)
\cong \mathfrak{u}_{q^6, 1}(\mathfrak{sl}_n)\cong \mathfrak{u}_{1, q^6}(\mathfrak{sl}_n);$& $4^{(n+2)(n-1)}$ \\ \hline
  $10$ & $\mathfrak{u}_{q, q^3}(\mathfrak{sl}_n)\cong \mathfrak{u}_{q^3, q}(\mathfrak{sl}_n)
\cong \mathfrak{u}_{q^5, q^7}(\mathfrak{sl}_n)\cong \mathfrak{u}_{q^7, q^5}(\mathfrak{sl}_n);$& $8^{(n+2)(n-1)}$ \\ \hline
  $11$ & $\mathfrak{u}_{q^2, q^4}(\mathfrak{sl}_n)\cong \mathfrak{u}_{q^4, q^2}(\mathfrak{sl}_n)
\cong \mathfrak{u}_{q^6, q^4}(\mathfrak{sl}_n)\cong \mathfrak{u}_{q^4, q^6}(\mathfrak{sl}_n);$& $4^{(n+2)(n-1)}$\\ \hline
  $12$ & $\mathfrak{u}_{q^3, q^5}(\mathfrak{sl}_n)\cong \mathfrak{u}_{q^5, q^3}(\mathfrak{sl}_n);$& $8^{(n+2)(n-1)}$ \\  \hline
  $13$ & $\mathfrak{u}_{q, q^7}(\mathfrak{sl}_n)\cong \mathfrak{u}_{q^7, q}(\mathfrak{sl}_n);$& $8^{(n+2)(n-1)}$ \\  \hline
  $14$ & $\mathfrak{u}_{q, q^5}(\mathfrak{sl}_n)\cong \mathfrak{u}_{q^5, q}(\mathfrak{sl}_n)
\cong \mathfrak{u}_{q^7, q^3}(\mathfrak{sl}_n)\cong \mathfrak{u}_{q^3, q^7}(\mathfrak{sl}_n);$& $8^{(n+2)(n-1)}$ \\  \hline
  $15$ & $\mathfrak{u}_{q^2, q^6}(\mathfrak{sl}_n)\cong \mathfrak{u}_{q^6, q^2}(\mathfrak{sl}_n);$& $4^{(n+2)(n-1)}$\\  \hline
  $16$ & $\mathfrak{u}_{q, q^{-1}}(\mathfrak{sl}_n)=\mathfrak{u}_{1, q^4}(\mathfrak{sl}_n)\cong \mathfrak{u}_{q^4, 1}(\mathfrak{sl}_n)=\mathfrak{u}_{ q^{-1},q}(\mathfrak{sl}_n).$& $2^{(n+2)(n-1)}$\\  \hline
\end{tabular}
\end{table}
\end{prop}
\begin{proof}
   Using Theorem \ref{thm refineA}, we conclude that candidates $1$--$8$ are pairwise non-isomorphic.
   Similarly, candidates $10$, $12$, and $13$  are pairwise non-isomorphic.
   Combining with Corollaries \ref{cor isoA} and \ref{cor2 isoA},
   we obtain that  candidates $1$--$8$, $10$, $12$, $13$, and $14$ are non-isomorphic to each other.
   Obviously, candidates $9$, $11$, $15$, and $16$ are non-isomorphic to each other by Proposition \ref{prop A 4}
   and their distinguished dimensions further set them apart from the other candidates.
\end{proof}

\begin{remark}
  For the higher order of $q$, especially when $q$ is not prime, we can also get the complete classification of $\mathfrak{u}_{r, s}(\mathfrak{sl}_n)$ with the use of the refined Isomorphism Reduction Theorem (Theorem \ref{thm refineA}).
\end{remark}

\smallskip
\noindent{\it 3.2. Type $B$\;}
(1) Assume that $q$ is a primitive $5$th root of unity.
According to the Isomorphism Reduction Theorem of type $B$, we conclude that
\begin{prop}\label{thm B 5}
  Assume that $r^5=s^5=1$.
  Then Table \ref{table B 5} provides a complete classification of $\mathfrak{u}_{r, s}(\mathfrak{so}_{2n+1})$.
\end{prop}
\begin{table}
  \centering
  \caption{Isoclasses of type $B_n$ small quantum groups when $q$ is a primitive $5$th root of unity}\label{table B 5}
  \begin{tabular}{|c|l|c|l|}
    \hline
    1 & $\mathfrak{u}_{1, q}(\mathfrak{so}_{2n+1})\cong \mathfrak{u}_{q, 1}(\mathfrak{so}_{2n+1});$
   &6 & $\mathfrak{u}_{q, q^3}(\mathfrak{so}_{2n+1})\cong \mathfrak{u}_{q^3, q}(\mathfrak{so}_{2n+1});$ \\ \hline
    2 & $\mathfrak{u}_{1, q^2}(\mathfrak{so}_{2n+1})\cong \mathfrak{u}_{q^2, 1}(\mathfrak{so}_{2n+1});$
   &7 & $\mathfrak{u}_{q, q^4}(\mathfrak{so}_{2n+1})\cong \mathfrak{u}_{q^4, q}(\mathfrak{so}_{2n+1});$ \\ \hline
    3 & $\mathfrak{u}_{1, q^3}(\mathfrak{so}_{2n+1})\cong \mathfrak{u}_{q^3, 1}(\mathfrak{so}_{2n+1});$
   &8 & $\mathfrak{u}_{q^2, q^3}(\mathfrak{so}_{2n+1})\cong \mathfrak{u}_{q^3, q^2}(\mathfrak{so}_{2n+1});$\\ \hline
    4 & $\mathfrak{u}_{1, q^4}(\mathfrak{so}_{2n+1})\cong \mathfrak{u}_{q^4, 1}(\mathfrak{so}_{2n+1});$
   & 9 & $\mathfrak{u}_{q^2, q^4}(\mathfrak{so}_{2n+1})\cong \mathfrak{u}_{q^4, q^2}(\mathfrak{so}_{2n+1});$\\ \hline
    5 & $\mathfrak{u}_{q, q^2}(\mathfrak{so}_{2n+1})\cong \mathfrak{u}_{q^2, q}(\mathfrak{so}_{2n+1});$
   &10 & $\mathfrak{u}_{q^3, q^4}(\mathfrak{so}_{2n+1})\cong \mathfrak{u}_{q^4, q^3}(\mathfrak{so}_{2n+1}).$\\ \hline
  \end{tabular}
\end{table}

(2) More generally, assume that $q$ is a primitive $p$th root of unity, where $p$ is an odd prime with $p>3$.
By the Isomorphism Reduction Theorem of type $B$, we conclude that
    \begin{theorem}\label{thm B general}
    Assume that $r^p=s^p=1$, where $p$ is an odd prime with $p>3$.
    Then $\mathfrak{u}_{r, s}(\mathfrak{so}_{2n+1})\cong \mathfrak{u}_{r', s'}(\mathfrak{so}_{2n+1})$
    if and only if $(r, s)=(r', s')$ or $(r, s)=(s', r' )$.
    In this case, there are $\frac{p^2-p}{2}$ isoclasses.
    \end{theorem}

\begin{remark}
  Assume that $r$ is a primitive $d$th root of unity,
  $s$ is a primitive $d'$th root of unity and $\ell$ is the least common multiple of $d$ and $d'$.
  In the definition of $\mathfrak{u}_{r, s}(\mathfrak{so}_{2n+1})$ \cite{Hu  Wang  JGP 2010},
  we impose the conditions that $r^3\neq s^3, \;r^4\neq s^4$, and $\ell$ is an odd number.
While $\ell$ is typically assumed to be an odd prime, we can extend this to include composite values such as $\ell=9, 15, 21,\cdots$. Leveraging the refined Isomorphism Reduction Theorem and the convex PBW Lyndon Basis Theorem, we can derive many more isoclasses.
\end{remark}

\smallskip
\noindent{\it 3.3. Type $C$\;}

(1) Assume that $q$ is a primitive $5$th root of unity.
According to the Isomorphism Reduction Theorem of type $C$,
we conclude that
\begin{prop}\label{thm C 5}
  Assume that $r^5=s^5=1$.
  Then Table \ref{table C 5} provides a complete classification of $\mathfrak{u}_{r, s}(\mathfrak{sp}_{2n})$.
 \begin{table}
  \centering
  \caption{Isoclasses of type $C_n$ small quantum groups when $q$ is a primitive $5$th root of unity}\label{table C 5}
\begin{tabular}{|c|l|c|l|c|l|c|l|}
  \hline
   1 & $\mathfrak{u}_{1, q}(\mathfrak{sp}_{2n});$& 2 &$\mathfrak{u}_{1, q^2}(\mathfrak{sp}_{2n});$
  &3 &$\mathfrak{u}_{1, q^3}(\mathfrak{sp}_{2n});$ & 4& $\mathfrak{u}_{1, q^4}(\mathfrak{sp}_{2n});$\\ \hline
  5 & $\mathfrak{u}_{q, 1}(\mathfrak{sp}_{2n});$ & 6& $\mathfrak{u}_{q, q^2}(\mathfrak{sp}_{2n});$
  &7 & $\mathfrak{u}_{q, q^3}(\mathfrak{sp}_{2n});$ & 8& $\mathfrak{u}_{q, q^4}(\mathfrak{sp}_{2n});$ \\ \hline
  9 &$\mathfrak{u}_{q^2, 1}(\mathfrak{sp}_{2n});$ &10 &$\mathfrak{u}_{q^2, q}(\mathfrak{sp}_{2n});$
  &11& $\mathfrak{u}_{q^2, q^3}(\mathfrak{sp}_{2n});$ &12& $\mathfrak{u}_{q^2, q^4}(\mathfrak{sp}_{2n});$ \\  \hline
  13 &$\mathfrak{u}_{q^3, 1}(\mathfrak{sp}_{2n});$ &14& $\mathfrak{u}_{q^3, q}(\mathfrak{sp}_{2n});$
  &15 &$\mathfrak{u}_{q^3, q^2}(\mathfrak{sp}_{2n});$  &16& $ \mathfrak{u}_{q^3, q^4}(\mathfrak{sp}_{2n});$\\ \hline
  17 &$\mathfrak{u}_{q^4, 1}(\mathfrak{sp}_{2n});$ &18& $\mathfrak{u}_{q^4, q}(\mathfrak{sp}_{2n});$
  &19 &$\mathfrak{u}_{q^4, q^2}(\mathfrak{sp}_{2n});$ &20& $ \mathfrak{u}_{q^4, q^3}(\mathfrak{sp}_{2n}).$\\ \hline
\end{tabular}
\end{table}
\end{prop}

(2) More generally, assume that $q$ is a primitive $p$th root of unity, where $p$ is an odd prime with $p>3$.
By the Isomorphism Reduction Theorem of type $C$,
we conclude that
\begin{theorem} \label{thm C general}
 Assume that $r^p=s^p=1$, where $p$ is an odd prime with $p>3$.
 Then $\mathfrak{u}_{r, s}(\mathfrak{sp}_{2n})\cong \mathfrak{u}_{r', s'}(\mathfrak{sp}_{2n})$
 if and only if $(r, s)=(r', s')$.
 In this case, there are $p^2-p$ isoclasses.
\end{theorem}

\begin{remark}
 In the definition of $\mathfrak{u}_{r, s}(\mathfrak{sp}_{2n})$ \cite{R Chen PhD},
 we impose the conditions that $r^3\neq s^3, \;r^4\neq s^4$, and $\ell$ is an odd number.
 While $\ell$ is typically assumed to be an odd prime, we can extend this to include composite values such as
$\ell=9, 15, 21,\cdots$. Leveraging the refined Isomorphism Reduction Theorem and the convex PBW Lyndon Basis Theorem, we can derive many more isoclasses.
\end{remark}

\smallskip
\noindent{\it 3.4. Type $D$\;}
(1)  Assume that $q$ is a primitive $5$th root of unity.
Since $5$ is a prime, then for all of these candidates,
$rs^{-1}$ is a primitive $5$th root of unity.
According to the Isomorphism Reduction Theorem, we have the following
\begin{prop}\label{thm D 5}
Assume that $r^5=s^5=1$.
Then Table \ref{table D 5} provides a complete classification of $\mathfrak{u}_{r, s}(\mathfrak{so}_{2n})$.
\begin{table}[h]
  \centering
  \caption{Isoclasses of type $D$ small quantum groups when $q$ is a primitive $5$th root of unity}\label{table D 5}
\begin{tabular}{|c|l|c|l|}
  \hline
   1 & $\mathfrak{u}_{1, q}(\mathfrak{so}_{2n})\cong \mathfrak{u}_{q^4, 1}(\mathfrak{so}_{2n});$
  &7 &$\mathfrak{u}_{1, q^3}(\mathfrak{so}_{2n})\cong \mathfrak{u}_{q^2, 1}(\mathfrak{so}_{2n});$\\ \hline
  2 & $\mathfrak{u}_{q, q^2}(\mathfrak{so}_{2n})\cong \mathfrak{u}_{q^3, q^4}(\mathfrak{so}_{2n});$
  &8 & $\mathfrak{u}_{q, q^4}(\mathfrak{so}_{2n})=\mathfrak{u}_{q,q^{-1}}(\mathfrak{so}_{2n});$\\ \hline
  3 &$\mathfrak{u}_{q^2, q^3}(\mathfrak{so}_{2n});$
  &9 &$\mathfrak{u}_{q^3, q}(\mathfrak{so}_{2n})\cong \mathfrak{u}_{q^4, q^2}(\mathfrak{so}_{2n});$\\  \hline
  4 &$\mathfrak{u}_{1, q^2}(\mathfrak{so}_{2n})\cong \mathfrak{u}_{q^3, 1}(\mathfrak{so}_{2n});$
  &10 &$\mathfrak{u}_{1, q^4}(\mathfrak{so}_{2n})\cong \mathfrak{u}_{q, 1}(\mathfrak{so}_{2n});$\\ \hline
  5 &$\mathfrak{u}_{q, q^3}(\mathfrak{so}_{2n})\cong \mathfrak{u}_{q^2, q^4}(\mathfrak{so}_{2n});$
  &11 &$\mathfrak{u}_{q^2, q}(\mathfrak{so}_{2n})\cong \mathfrak{u}_{q^4, q^3}(\mathfrak{so}_{2n});$\\ \hline
  6 &$\mathfrak{u}_{q^4, q}(\mathfrak{so}_{2n})=\mathfrak{u}_{q^{-1}, q}(\mathfrak{so}_{2n});$
  &12&$\mathfrak{u}_{q^3, q^2}(\mathfrak{so}_{2n}).$\\ \hline
\end{tabular}
\end{table}
\end{prop}

(2) More generally, when $q$ is a primitive root of unity with odd prime order $p$ with $p>3$.
By the Isomorphism Reduction Theorem of type $D$, we have the following
 \begin{theorem}\label{thm D general}
 Assume that $r^p=s^p=1$, where $p$ is an odd prime with $p>3$.
 Then Table \ref{table D general} gives the complete classification of $\mathfrak{u}_{r, s}(\mathfrak{so}_{2n})$.
 In this case, there are $\frac{p^2-1}{2}$ isoclasses.
\begin{table}[h]
  \centering
  \caption{Isoclasses of type $D$ small  quantum groups when $q$ is a primitive root of unity with odd prime order $p$ }\label{table D general}
  \begin{tabular}{|l|}
    \hline
      $\mathfrak{u}_{q^k, q^{k+t}}(\mathfrak{so}_{2n})\cong \mathfrak{u}_{q^{p-k-t}, q^{p-k}}(\mathfrak{so}_{2n});\;
      (0<t\leq p-1, 0\leq k \leq p-1, t, k\in \mathbb{Z})$\\ \hline
  \end{tabular}
  \end{table}
\end{theorem}

\begin{remark}
 In the definition of $\mathfrak{u}_{r, s}(\mathfrak{so}_{2n})$ \cite{Bai Xiaotang 2008},
  we impose the conditions that $r^2\neq s^2$, and $\ell$ is an odd number.
While $\ell$ is typically assumed to be an odd prime, we extend this to include composite values such as
 $\ell=9, 15, 21,\cdots$. Leveraging the refined Isomorphism Reduction Theorem and the convex PBW Lyndon Basis Theorem, we can derive many more isoclasses.
\end{remark}

\smallskip
\noindent{\it 3.5. Type $F_4$\;}
(1) Assume that $q$ is a primitive $5$th root of unity.
According to the Isomorphism Reduction Theorem of type $F_4$, we conclude that
\begin{prop}\label{thm F 5}
Assume that $r^5=s^5=1$.
Then Table \ref{table F 5} provides a complete classification of $\mathfrak{u}_{r, s}(\mathfrak{so}_{2n})$.
\begin{table}[h]
  \centering
  \caption{Isoclasses of type $F_4$ small quantum groups when $q$ is a primitive $5$th root of unity}\label{table F 5}
\begin{tabular}{|c|l|c|l|}
  \hline
  1 & $\mathfrak{u}_{1, q}(F_4)\cong \mathfrak{u}_{q, 1}(F_4);$
  &6 & $\mathfrak{u}_{q, q^3}(F_4)\cong \mathfrak{u}_{q^3, q}(F_4);$\\ \hline
  2 & $\mathfrak{u}_{q, q^2}(F_4)\cong \mathfrak{u}_{q^2, q}(F_4);$
  &7 & $\mathfrak{u}_{q^2, q^4}(F_4)\cong \mathfrak{u}_{q^4, q^2}(F_4);$\\ \hline
  3 & $\mathfrak{u}_{q^3, q^2}(F_4)\cong \mathfrak{u}_{q^2, q^3}(F_4);$
  &8& $\mathfrak{u}_{q^3, 1}(F_4)\cong \mathfrak{u}_{1, q^3}(F_4);$  \\  \hline
  4 & $\mathfrak{u}_{q^4, q^3}(F_4)\cong \mathfrak{u}_{q^3, q^4}(F_4);$
  &9 & $\mathfrak{u}_{q^4, q}(F_4)\cong \mathfrak{u}_{q, q^4}(F_4);$\\ \hline
  5 & $\mathfrak{u}_{1, q^2}(F_4)\cong \mathfrak{u}_{q^2, 1}(F_4);$
  &10 & $\mathfrak{u}_{q^4, 1}(F_4)\cong \mathfrak{u}_{1, q^4}(F_4).$\\ \hline
\end{tabular}
\end{table}
\end{prop}

(2) More generally, assume that $q$ is a primitive $p$th root of unity,
where $p$ is an odd prime with $p>3$.
By the Isomorphism Reduction Theorem of type $F_4$, we conclude that
  \begin{theorem}\label{thm F general}
 Assume that $r^p=s^p=1$, where $p$ is an odd prime with $p>3$.
  Then $\mathfrak{u}_{r, s}(F_4)\cong \mathfrak{u}_{r', s'}(F_4)$
if and only if $(r, s)=(r', s')$ or $(r, s)=(s', r' )$.
As a result, there are $\frac{p^2-p}{2}$ isoclasses.
\end{theorem}

\begin{remark}
   In the definition of $\mathfrak{u}_{r, s}(F_4)$ \cite{Chen Hu and Wang to appear},
   we impose the conditions that $ r^3\neq s^3, \;r^4\neq s^4$, and $\ell$ is an odd number.
   Additionally, while $\ell$ is typically assumed to be an odd prime, we extend this to include composite values such as
$\ell=9, 15, 21,\cdots$. In principle, leveraging the refined Isomorphism Reduction Theorem, we are able to derive many more isoclasses. However, due to the enormous computational complexity, we are confined to our current assumptions.
\end{remark}

\smallskip
\noindent{\it 3.6. Type $G_2$\;}
(1)  Assume $q$ is a primitive $p$th root of unity, where $p$ is an odd prime with $p>3$.
According to the Isomorphism Reduction Theorem of type $G_2$, we conclude that
\begin{theorem}\label{thm G general}
   Assume that $r^p=s^p=1$, where $p$ is an odd prime with $p>3$.
    Then $\mathfrak{u}_{r, s}(G_2)\cong \mathfrak{u}_{r', s'}(G_2)$
if and only if $(r, s)=(r', s')$ or $(r, s)=(s', r' )$.
As a result, there are $\frac{p^2-p}{2}$ isoclasses.
\end{theorem}

(2) Assume that $q$ is an primitive $8$th root of unity,
using the refined Isomorphism Reduction Theorem of type $G_2$, we have the following
\begin{theorem}\label{thm G 8}
Assume that $r^8=s^8=1$.
Then table \ref{table G 8} provides a complete classification of $\mathfrak{u}_{r, s}(G_2)$.
\begin{table}[h]
  \centering
  \caption{Isoclasses of type $G_2$ small quantum groups when $q$ is a primitive $8$th root of unity}\label{table G 8}
\begin{tabular}{|c|l|c|l|}
  \hline
   1 & $\mathfrak{u}_{1, q}(G_2)\cong \mathfrak{u}_{q, 1}(G_2);$
  &9 & $\mathfrak{u}_{1, q^3}(G_2)\cong \mathfrak{u}_{q^3, 1}(G_2);$ \\ \hline
  2 & $\mathfrak{u}_{q, q^2}(G_2)\cong \mathfrak{u}_{q^2, q}(G_2);$
  &10 & $\mathfrak{u}_{q, q^4}(G_2)\cong \mathfrak{u}_{q^4, 1}(G_2);$ \\ \hline
  3 & $\mathfrak{u}_{q^2, q^3}(G_2)\cong \mathfrak{u}_{q^3, q^2}(G_2);$
  &11 & $\mathfrak{u}_{q^2, q^5}(G_2)\cong \mathfrak{u}_{q^5, q^2}(G_2);$ \\  \hline
  4 & $\mathfrak{u}_{q^3, q^4}(G_2)\cong \mathfrak{u}_{q^4, q^3}(G_2);$
  &12 & $\mathfrak{u}_{q^3, q^6}(G_2)\cong \mathfrak{u}_{q^6, q^3}(G_2);$\\ \hline
  5 & $\mathfrak{u}_{q^4, q^5}(G_2)\cong \mathfrak{u}_{q^5, q^4}(G_2);$
  &13 & $\mathfrak{u}_{q^4, q^7}(G_2)\cong \mathfrak{u}_{q^7, q^4}(G_2);$ \\ \hline
  6 & $\mathfrak{u}_{q^5, q^6}(G_2)\cong \mathfrak{u}_{q^6, q^5}(G_2);$
  &14 & $\mathfrak{u}_{q^5, 1}(G_2)\cong \mathfrak{u}_{1, q^5}(G_2);$\\ \hline
  7 & $\mathfrak{u}_{q^6, q^7}(G_2)\cong \mathfrak{u}_{q^7, q^6}(G_2);$
  &15 & $\mathfrak{u}_{q^6, q}(G_2)\cong \mathfrak{u}_{q, q^6}(G_2);$ \\ \hline
  8 & $\mathfrak{u}_{q^7, 1}(G_2)\cong \mathfrak{u}_{1, q^7}(G_2);$
  &16 & $\mathfrak{u}_{q^7, q^2}(G_2)\cong \mathfrak{u}_{q^2, q^7}(G_2).$ \\ \hline
\end{tabular}
\end{table}
\end{theorem}

\begin{remark}
In the definition of $\mathfrak{u}_{r, s}(G_2)$ \cite{Hu Wang  Pacific J. Math 2009},
we assume that $ r^4\neq s^4, r^6\neq s^6$, and $\ell $ is coprime to $3$.
Thus, while $\ell$ is typically assumed to be an odd prime,
we can extend this to include composite values such as $\ell=10, 14, \cdots$.
Leveraging the refined Isomorphism Reduction Theorem, we can derive many more isoclasses.
\end{remark}

\section{Alternative Approach to Reaching the Classification Results: via Radford's Classification of Simple Yetter-Drinfeld Modules}

As an effective verification of the correctness of our aforementioned classification results, we adopt a representation-theoretic approach here: Specifically, we use the characterization of the dimension distribution of simple Yetter-Drinfeld modules due to Radford for a class of pointed Hopf algebras. In this way, as some illustrative examples, we obtain the same classification results for type $A_2$ with lower orders $\{4, 6, 8\}$.

Assume $r=q^x, s=q^y$, where $q$ is a primitive $\ell$th root of unity.
In this section, we focus on the cases where the order of parameter $q$ is $\ell\in\{4, 6, 8\}$.

\smallskip
\noindent{\it 4.1. \;}
Assume that $q$ is a primitive $4$th root of unity.
In the previous section, we have used Theorem \ref{thm refineA} to get a complete classification of $\mathfrak{u}_{r,s}(\mathfrak{sl}_n)$, see Table \ref{table A 4}.
Here we will employ representation theory to get the same result.
First, note that Benkart and Whiterspoon established a sufficient condition for determining whether $\mathfrak{u}_{r,s}(\mathfrak{sl}_n)$ possesses a Drinfeld double structure.
\begin{theorem}\cite{Benkart Witherspoon Fields Institute Communications 2004} \label{thm double structure}
  Assume that $(y^{n-1}-y^{n-2}z+-\cdots +(-1)^{n-1}z^{n-1}, \ell)=1$,
  and let $\mathfrak{b}$ be the subalgebra of the small quantum groups $\mathfrak{u}_{r, s}(\mathfrak{sl}_n)$
  generated by the elements $\omega_i, e_i, \; 1\leq i<n$.
  There is an isomorphism of Hopf algebras $D(\mathfrak{b})\cong \mathfrak{u}_{r, s}(\mathfrak{sl}_n)$.
\end{theorem}
When $n=3$, the candidate $3$ in Table \ref{table A 4} does not satisfy Theorem \ref{thm double structure},
so we do not know whether it has double structure or not.
Therefore we only need to consider the remaining three candidates.
Then we use the twist equivalence theorem (\cite{Benkart Pereira Witherspoon 2010} Theorem 4.12),
we find that the candidates $1, 2$ have the same dimension distribution.
So we reduce to the following candidates: $\mathfrak{u}_{1, q}(\mathfrak{sl}_3), \mathfrak{u}_{q, q^{-1}}(\mathfrak{sl}_3)$.
Their dimension distribution of simple modules (i.e., simple Yetter-Drinfeld modules) was calculated in \cite{Benkart Pereira Witherspoon 2010} are different.
The following result is due to \cite{Benkart Pereira Witherspoon 2010}.

\begin{prop}\label{prop A 4 new}
   Assume that $q$ is a primitive $4$th root of unity, then
   $$\mathfrak{u}_{1, q}(\mathfrak{sl}_3)\ncong \mathfrak{u}_{q, q^{-1}}(\mathfrak{sl}_3).$$
\end{prop}

\begin{remark}
  Although we have found that $\mathfrak{u}_{1, q}(\mathfrak{sl}_3)\ncong \mathfrak{u}_{q, q^2}(\mathfrak{sl}_3)$ directly from Theorem \ref{thm refineA}, it is worthwhile to mention that they have the same dimension distribution of their simple modules.
\end{remark}

\begin{remark}
  We have proved $\mathfrak{u}_{1, q}(\mathfrak{sl}_3)\ncong \mathfrak{u}_{q, q^{-1}}(\mathfrak{sl}_3)$ when $q$ is a primitive $4$th root of unity, via calculating their dimension distribution of their modules.
  For general $n$, this method seems more complicated.
  However, we can get a more general conclusion
  $$\mathfrak{u}_{1, q}(\mathfrak{sl}_n)\ncong \mathfrak{u}_{q, q^{-1}}(\mathfrak{sl}_n)$$
  when $q$ is a primitive $4$th root of unity, by Corollary \ref{cor isoA}.
\end{remark}

\vspace{1em}
\noindent{\it 4.2. \;}
Assume that $q$ is a primitive $6$th root of unity.
In the previous section, we have used Theorem \ref{thm refineA} to get Table \ref{table A 6}.
When $n=3$, the candidates $4, 8$ do not satisfy Theorem \ref{thm double structure},
so we do not know whether they have double structure or not.
Therefore we only need to consider the remaining seven candidates.
Then we use the twist equivalence theorem,
we find that candidates $1, 2, 3$ have the same dimension distribution.
So it suffices to consider the following four candidates:
$\mathfrak{u}_{1, q}(\mathfrak{sl}_3), \mathfrak{u}_{q^2, q^5}(\mathfrak{sl}_3), \mathfrak{u}_{q, q^3}(\mathfrak{sl}_3),
\mathfrak{u}_{q, q^5}(\mathfrak{sl}_3)$.

Finally, we calculate the simple Yetter-Drinfeld modules distribution of these candidates respectively:

\vspace{0.5em}
 \noindent$(1)$ $\dim(H_{1, q}\bullet_\beta g), g\in G(H_{1, q}), \beta\in \widehat{G(H_{1, q})}:  $

\begin{gather*}  \{\, 1^{36}, 3^{72}, 6^{72}, 8^{36},  10^{72}, 15^{144}, 24^{72}, 25^{72}, 27^{108}, 48^{72}, \\
54^{72}, 56^{36}, 87^{72}, 120^{72}, 124^{36}, 165^{72}, 216^{36}\, \};
\end{gather*}

\noindent
$(2)$ $\dim(H_{q^2, q^5}\bullet_\beta g), g\in G(H_{q^2, q^5}), \beta\in \widehat{G(H_{q^2, q^5})}:$

$$\{\,1^{36}, 3^{72}, 8^{36}, 36^{432}, 72^{432}, 216^{288}\,\};$$

\noindent
$(3)$ $\dim(H_{q, q^3}\bullet_\beta g), g\in G(H_{q^3, q^5}), \beta\in \widehat{G(H_{q, q^3})}:$

$$\{\,1^{36}, 3^{72}, 6^{72}, 7^{36}, 15^{72}, 27^{36}, 36^{324}, 72^{324}, 108^{324}\,\};$$

\noindent
$(4)$ $\dim(H_{q, q^5}\bullet_\beta g), g\in G(H_{q^3, q^5}), \beta\in \widehat{G(H_{q, q^3})}:$

$$\{\,1^{36}, 3^{72}, 6^{72}, 7^{36}, 15^{72}, 27^{36}, 36^{324}, 72^{324}, 108^{324}\,\}.$$
Here $1^{16}$ means that there are $16$ simple modules whose dimension are $1$.
Hence, we obtain these pairwise non-isomorphic small quantum groups:
$\mathfrak{u}_{1, q}(\mathfrak{sl}_3)$, $\mathfrak{u}_{q^2, q^5}(\mathfrak{sl}_3)$, $\mathfrak{u}_{q, q^3}(\mathfrak{sl}_3)$.
Thus, we have
\begin{prop}\label{prop A6new}
  Assume that $q$ is a primitive $6$th root of unity,
  then $$\mathfrak{u}_{1, q}(\mathfrak{sl}_3), \quad \mathfrak{u}_{q^2, q^5}(\mathfrak{sl}_3)  \quad \text{and} \quad\mathfrak{u}_{q, q^3}(\mathfrak{sl}_3)$$ are pairwise non-isomorphic.
\end{prop}

\smallskip
\noindent{\it 4.3. \;}
Assume that $q$ is an primitive $8$th root of unity.
In the previous section, we have used Theorem \ref{thm refineA} to get Table \ref{table A 8}.
When $n=3$, the $16$th candidate does not satisfy Theorem \ref{thm double structure},
so we do not know whether it has double structure or not.
Therefore we only need to consider the remaining $15$ candidates.
Similarly, we find that the candidates $1$--$7$ have the same dimension distribution;
and the candidates $10$ and $12$ have the same dimension distribution.
So it suffices to consider the following candidates:
$$\mathfrak{u}_{1, q}(\mathfrak{sl}_3), \quad \mathfrak{u}_{q, q^3}(\mathfrak{sl}_3), \quad \mathfrak{u}_{q, q^5}(\mathfrak{sl}_3), \quad
\mathfrak{u}_{q, q^6}(\mathfrak{sl}_3), \quad \mathfrak{u}_{q, q^7}(\mathfrak{sl}_3).$$

Finally,  we calculate the dimension distributions of these candidates, respectively:

\vspace{0.5em}
 \noindent$(1)\;\dim(H_{1, q}\bullet_\beta g), g\in G(H_{1, q}), \beta\in \widehat{G(H_{1, q})}:  $

 \vspace{1em}
  \noindent$  \{1^{64}, 3^{128}, 6^{128}, 8^{64},  10^{128}, 15^{256}, 24^{128}, 27^{64}, 28^{128}, 35^{128}, 36^{128}, 42^{256}, 46^{128}, 48^{256},$

  \vspace{1em}
   \noindent$60^{128}, 64^{64}, 80^{128}, 90^{128}, 96^{128}, 98^{64}, 132^{128}, 144^{128},
  150^{128}, 192^{128}, 204^{128}, 260^{128}, $

  \vspace{1em}
  \noindent$270^{128}, 336^{128}, 342^{64}, 420^{128}, 512^{64} \}; $

   \vspace{1em}
  \noindent$(2)\;\dim(H_{q, q^6}\bullet_\beta g), g\in G(H_{q, q^6}), \beta\in \widehat{G(H_{q, q^6})}:  $

 \vspace{1em}
  \noindent$  \{1^{64}, 3^{128}, 6^{128}, 8^{64},  10^{128}, 15^{256}, 24^{128}, 27^{64}, 28^{128}, 35^{128}, 36^{128}, 42^{256}, 46^{128}, 48^{256},$

  \vspace{1em}
   \noindent$60^{128}, 64^{64}, 80^{128}, 90^{128}, 96^{128}, 98^{64}, 132^{128}, 144^{128},
  150^{128}, 192^{128}, 204^{128}, 260^{128}, $

  \vspace{1em}
  \noindent$270^{128}, 336^{128}, 342^{64}, 420^{128}, 512^{64} \}; $

  \vspace{1em}
  \noindent$(3)\;\dim(H_{q, q^3}\bullet_\beta g), g\in G(H_{q, q^3}), \beta\in \widehat{G(H_{q, q^3})}:$
  \vspace{1em}

  \noindent$\{1^{64}, 3^{128}, 6^{128}, 8^{64}, 10^{128}, 12^{128}, 24^{128},$
  $26^{64}, 42^{128}, 64^{128}, 128^{768}, 192^{768}, 256^{768}\}; $

  \vspace{1em}
  \noindent$(4)\;\dim(H_{q, q^7}\bullet_\beta g), g\in G(H_{q, q^7}), \beta\in \widehat{G(H_{q, q^7})}:$
  \vspace{1em}

  \noindent$\{1^{64}, 3^{128}, 6^{128}, 8^{64}, 10^{128}, 12^{128}, 24^{128},$
  $26^{64}, 42^{128}, 64^{128}, 128^{768}, 192^{768}, 256^{768}\}; $

  \vspace{1em}
  \noindent$(5)\;\dim(H_{q, q^5}\bullet_\beta g), g\in G(H_{q, q^5}), \beta\in \widehat{G(H_{q, q^5})}:$

  \vspace{1em}
  \noindent$\{1^{64}, 3^{128},  8^{64}, 27^{36}, 64^{1152}, 128^{1152},
   512^{1536}\}.$
 \vspace{1em}

In summary, we obtain the next Proposition.
\begin{prop}\label{prop A8new}
   Assume that $q$ is a primitive $8$th root of unity, then the following three small quantum groups are pairwise non-isomorphic:
  $$\mathfrak{u}_{1, q}(\mathfrak{sl}_3), \quad \mathfrak{u}_{q, q^5}(\mathfrak{sl}_3), \quad \text{and}\quad \mathfrak{u}_{q, q^{-1}}(\mathfrak{sl}_3).$$
\end{prop}

\begin{remark}
If we assume $q^2=t$, then $t$ is a primitive $4$th root of unity.
As shown in the previous subsection,
$\mathfrak{u}_{1, q^2}(\mathfrak{sl}_3), \mathfrak{u}_{q^2, q^4}(\mathfrak{sl}_3),
 \mathfrak{u}_{q^2, q^6}(\mathfrak{sl}_3)$ have a double structure despite not satisfying Theorem \ref{thm double structure}.
\end{remark}

\begin{remark}
  In this section, we only consider type $A_2$.
  However, for type $A_3$, many candidates do not have Drinfeld double structure. By the way, this property is one of the key features that the HP Hopf algebras Beliakova's team is seeking in their studies on non-semisimple TQFT should possess.
  For types $B, C, D, F_4, G_2$, while $\ell\geq 8$ is theoretically plausible, it becomes computationally infeasible due to excessive demands for large $\ell$.
\end{remark}

\begin{remark}
For orders $4\le \ell\le 8$ of parameter $q$, In addition to the $45$ standard one-parameter small quantum groups,
 we identify $209$ new exotic one-parameter non-standard small quantum group isoclasses across all types:
 $47$ for type $A$,
$31$  for type $B$,
 $62$ for type $C$,
 $36$ for type $D$,
 $31$ for type $F_4$,
 $47$ for type $G_2$.
\end{remark}

\textbf{Acknowledgments}\; We are grateful to the referee for his/her valuable comments.
In the process of reorganizing our rigorous proof, we have found that it is possible to provide a complete and neat classification method and corresponding results only via the refined Isomorphism Reduction Theorem and the convex PBW Lyndon basis Theorem, for instance, see Propositions \ref{prop A 4}, \ref{prop A 6} and \ref{prop A 8}.

\vskip15pt

\bibliographystyle{amsalpha}

\end{document}